\documentclass[reqno]{amsart}
\usepackage{latexsym,amsmath,amssymb,amscd, amsthm}
\usepackage[utf8]{inputenc}
\usepackage[all]{xy}
\usepackage{enumerate}

\usepackage{color}
\usepackage{hyperref} 
\makeatletter
\@addtoreset{figure}{section}
\def\thefigure{\thesection.\@arabic\c@figure}
\def\fps@figure{h,t}
\@addtoreset{table}{bsection}

\def\thetable{\thesection.\@arabic\c@table}
\def\fps@table{h, t}
\@addtoreset{equation}{section}

\makeatother

\theoremstyle{plain}
\newtheorem{theorem}{Theorem}
\newtheorem*{theorem*}{Theorem}
\newtheorem{corollary}[theorem]{Corollary}

\newtheorem{example}[theorem]{Example}
\newtheorem{lemma}[theorem]{Lemma}

\newtheorem{proposition}[theorem]{Proposition}
\newtheorem{remark}[theorem]{Remark}

\numberwithin{theorem}{section}
\numberwithin{equation}{section}

\renewcommand{\1}{{\bf 1}}

\newcommand{\Ad}{{\rm Ad}}
\newcommand{\ad}{{\rm ad}}

\newcommand{\Ci}{{\mathcal C}^\infty}

\newcommand{\Aut}{{\rm Aut}}
\newcommand{\Bun}{\text{{\boldmath{$\mathfrak{B}$}}}}
\newcommand{\Der}{{\rm Der}}

\newcommand{\de}{{\rm d}}

\newcommand{\Galg}{G^{\rm alg} }
\newcommand{\ggalg}{\gg^{\rm alg}}

\newcommand{\I}{{\rm I}}
\newcommand{\II}{{\rm II}}
\newcommand{\III}{{\rm III}}

\newcommand{\Ind}{{\rm Ind}}
\newcommand{\ind}{{\rm ind}}

\newcommand{\nor}{{\rm nor}}

\renewcommand{\O}{{\mathbf{O}}}

\newcommand{\Prim}{{\rm Prim}}

\newcommand{\RelS}{\text{{\boldmath{$\mathcal{S}$}}}}

\newcommand{\spa}{{\rm span}\,}

\newcommand{\Tr}{{\rm Tr}\,}

\newcommand{\CC}{{\mathbb C}}

\newcommand{\KK}{{\mathbb K}}

\newcommand{\RR}{{\mathbb R}}

\newcommand{\Bc}{{\mathcal B}}

\newcommand{\Oc}{{\mathcal O}}

\newcommand{\Vc}{{\mathcal V}}

\newcommand{\Gg}{{\mathfrak G}}

\newcommand{\ag}{{\mathfrak a}}

\newcommand{\dg}{{\mathfrak d}}
\renewcommand{\gg}{{\mathfrak g}}
\newcommand{\hg}{{\mathfrak h}}

\renewcommand{\ng}{{\mathfrak n}}

\newcommand{\pg}{{\mathfrak p}}
\newcommand{\qg}{{\mathfrak q}}

\newcommand{\matt}[2]
{\ensuremath{\begin{pmatrix}
			{#1} & 0 \\
			0 & {#2}
\end{pmatrix}}}

\newcommand{\mattt}[3]
{\ensuremath{\begin{pmatrix}
			{#1} & & 0\\
			& {#2} & \\
			0& & {#3}
\end{pmatrix}}}

\usepackage{mathtools}

\newcommand{\mathsout}[1]
{\bgroup\mathchoice
	{\sbox0{$\displaystyle{#1}$}%
		\usebox0\hspace{-\wd0}%
		\rule[0.5\ht0-0.5\dp0-.5pt]{\wd0}{1pt}}%
	{\sbox0{$\textstyle{#1}$}%
		\usebox0\hspace{-\wd0}%
		\rule[0.5\ht0-0.5\dp0-.5pt]{\wd0}{1pt}}%
	{\sbox0{$\scriptstyle{#1}$}%
		\usebox0\hspace{-\wd0}%
		\rule[0.5\ht0-0.5\dp0-.5pt]{\wd0}{1pt}}%
	{\sbox0{$\scriptscriptstyle{#1}$}%
		\usebox0\hspace{-\wd0}%
		\rule[0.5\ht0-0.5\dp0-.5pt]{\wd0}{1pt}}%
	\egroup}

\title[Regular representation and open quasi-orbits]
{On the regular representation of solvable Lie groups with open coadjoint quasi-orbits}

\author{Ingrid Belti\c t\u a}
\author{Daniel Belti\c t\u a}
\address{Institute of Mathematics ``Simion Stoilow'' of the Romanian Academy,	P.O. Box 1-764, Bucharest, Romania}

\email{Ingrid.Beltita@imar.ro, ingrid.beltita@gmail.com}
\email{Daniel.Beltita@imar.ro, beltita@gmail.com}
\thanks{The research of the first-named author is based upon work from COST Action
	21109 CaLISTA, supported by COST (European Cooperation in
	Science and Technology), www.cost.eu.
	The research of the second-named author was supported by a grant of the  Ministry of Research, Innovation and Digitization, CNCS/CCCDI -- UEFISCDI, project number PN-III-P4-ID-PCE-2020-0878, within PNCDI III}

\begin{document}

	\begin{abstract}
		We obtain a Lie theoretic intrinsic characterization of the connected and simply connected solvable Lie groups whose regular representation is a  factor representation. 
		When this is the case,  
		the corresponding von Neumann algebras 
		are isomorphic to the hyperfinite $\II_\infty$ factor, and every Casimir function is constant. 
		We thus obtain a family of geometric models for the standard representation of that factor. 
	Finally, we show that the regular representation of any connected and simply connected solvable Lie group with open coadjoint orbits is always of type~$\I$, though the group needs not be of type~$\I$, and include some relevant examples.
	\end{abstract}

	\maketitle
	
	
	\section{Introduction}
In this paper, we obtain an intrinsic characterization of the connected and simply connected solvable Lie groups with the property that the group von Neumann algebra is a factor. 
Our construction leads to solvable Lie groups for which the regular representation is obtained by geometric quantization starting from a suitable Poisson manifold, specifically a symplectic foliation which is a generalized coadjoint orbit,  in the language of \cite{Pu86}.
	
	There is no available characterization so far of the locally compact groups whose regular representation is a factor representation, despite the interest in this problem, as shown for instance by the various disparate results collected in \cite[Sect.~7.E]{BkH20}. 
It is known that 
	for any countable 
	discrete group, its regular representation is a factor representation if and only if the group under consideration has infinite conjugacy classes. 
	That simple characterization is  however specific to discrete groups. 
	For many connected Lie groups (including the compact, the nilpotent, and the reductive ones) the regular representation is never a factor representation. 
Nevertheless, we do find group von Neumann algebras that are factors in the class of (connected and simply connected)  solvable Lie groups.

However, unlike the case of discrete groups, we find that in the case of solvable Lie groups, if the regular representation is a factor representation then the corresponding factor is always isomorphic to the hyperfinite type~\II$_\infty$ factor.
	We thus add to the previous geometric realizations of   the hyperfinite $\II_\infty$ factor (see, e.g., \cite{Dy95}), by providing geometric models associated to families of simply connected solvable Lie groups. 
	
	Our results have also  strong implications on the Lie-Poisson space 
	$\gg^*$, where $\gg$ is the Lie algebra of a solvable Lie group $G$. 
We recall that
a (generalized) Casimir function $c\in\Ci(\gg^*)$ by definition satisfies $\{c,\varphi\}=0$ for all $\varphi\in\Ci(\gg^*)$, where $\{\cdot, \cdot\}$ 
is the Lie-Poisson bracket;
equivalently, the smooth function $c\colon\gg^*\to\RR$ is 
constant on the coadjoint orbits 
of the Lie group $G$. 
While the space of polynomial Casimir functions for semisimple Lie algebras and even for nilpotent Lie algebras has been pretty well understood since a long time ago, much less information is available for Casimir functions on general solvable Lie algebras. 
Yet, there is a recent surge of interest in this area, in connection with some problems in mathematical physics and integrable systems, 
since the polynomial Casimir functions on $\gg^*$ parameterize the 2-sided differential operators on the Lie group~$G$, whose spectra characterize specific properties of physical systems that admit $G$ as a symmetry group. 
See for instance \cite{DS24} and the references therein.

A striking difference from the case of semisimple Lie algebras and nilpotent Lie algebras is the existence of solvable Lie algebras for which every Casimir function is constant. 
This is the case for instance for the so-called Frobenius Lie algebras, that is, 
solvable Lie algebras that admit open coadjoint orbits, 
but it turns out that these are by no means the only examples. 
In the present paper we approach this phenomenon from the perspective of representations theory. 
In particular, our systematic investigation of the factoriality property of the regular representation reveals a class of groups that do not have open coadjoint orbits and yet they satisfy the condition on absence of nonconstant Casimir functions.

	If $G$ is a solvable Lie group whose regular representation is a factor representation, then the group $C^*$-algebra $C^*(G)$ is necessarily primitive, that is, there exists a faithful irreducible $*$-representation of $C^*(G)$.
Thus, the present approach leads in particular to the first examples of non-type-$\I$ solvable Lie groups with primitive $C^*$-algebras. 
A systematic study of solvable Lie groups with this last property was initiated in \cite{BB18} using groupoid methods, however the specific examples illustrating the general theory were constructed within the theory of algebraic linear groups, 
which are well known to be necessarily type~$\I$. 
(See e.g., \cite{BeEc21} for a recent general result in this connection.) 
We recall that the method in \cite{BB18} of constructing faithful irreducible $*$-representations of group $C^*$-algebras relied on linear group actions
that have an open dense orbit.

	The structure of this paper is as follows: 
	In Section~\ref{Sect2} we collect a few remarks on the Puk\'ansky correspondence, for later use. 
	In Section~\ref{Sect3} we obtain the characterization of solvable Lie groups whose regular representation is a factor representation (Theorem~\ref{factor3}), and show that in that case, the group von Neumann algebra is isomorphic to the hyperfinite \II$_\infty$ factor.
(Corollary~\ref{typeI-fact}). 
As a consequence we obtain that these groups 
have no non-constant Casimir functions. 
	Finally, in Section~\ref{Sect4}, we study the regular representation of Frobenius Lie groups.
	We find that the corresponding  group von Neumann algebras 
are always type \I\ (Corollary~\ref{regI-cor}), but we prove by example that a 
Frobenius Lie group may not be type~$\I$ (Example~\ref{exF}).

\subsection*{General notation}
The 1-connected (that is, connected and simply connected)  Lie groups are denoted by upper case Roman letters and their Lie algebras by the corresponding
lower case Gothic letters. 

For a locally compact group $H$, let $\lambda_H\colon H\to \Bc(L^2(H))$  be its left regular representation. 
It yields a $*$-representation $\lambda_H\colon L^1(H)\to \Bc(L^2(H))$, where $L^1(H)$ is regarded as a Banach $*$-algebra with respect to the usual convolution. 
See for instance \cite{Di64} and \cite{BkH20} 
for additional background information on the group $C^*$-algebra $C^*(H)$ and its primitive ideal space $\Prim(H)$. 
Also, we refer to \cite{BC20} for representation theory of solvable Lie groups.

\section{Preliminaries}
\label{Sect2}

Throughout this paper, $G$ denotes a 1-connected solvable Lie group with its Lie algebra~$\gg$. We use the notation introduced in \cite[Sect. 2]{BB24}. 
We collect below a few basic ingredients of the Puk\'anszky correspondence.

The starting point is the partition of the dual of the Lie algebra into coadjoint quasi-orbits via the equivalence relation on $\gg^*$ defined by  $\xi\sim\eta$ if and only if $\overline{G\xi}=\overline{G\eta}$, where $\overline{G\xi}$ is the closure of the coadjoint orbit of~$\xi\in\gg^*$. 
We thus obtain a partition into immersed smooth manifolds 
$$\gg^*=\bigsqcup_{\Oc\in(\gg^*/G)^\sim}\Oc$$
and one then constructs a certain space $\Bun(\gg^*)$ equipped with a surjective, $G$-equivariant, intrinsically defined map $\tau\colon\Bun(\gg^*)\to\gg^*$ with the property that the restriction $\tau\vert_{\Bun(\Oc)}\colon\Bun(\Oc)\to\Oc$  is a trivial, smooth torus bundle 
for every $\Oc\in(\gg^*/G)^\sim$, where $\Bun(\Oc):=\tau^{-1}(\Oc)$. 
The total space $\Bun(\Oc)$ is partitioned into $G$-orbit closures called \emph{generalized coadjoint orbits} and their set is denoted by $(\Bun(\Oc)/G)^\approx$. 
The set of all generalized coadjoint orbits, when $\Oc$ runs over  $(\gg^*/G)^\sim$,  is denoted by~$\RelS$.

\subsection*{Elements of Puk\'anszky theory}
A key piece in the 
Puk\'anszky theory is the equivariance property of the Kirillov correspondence between coadjoint orbits and equivalence classes of unitary irreducible representations of any nilpotent Lie group (cf., e.g., \cite[page~603]{Pu71} or \cite[III, 1c, page 82]{Pu73}). 
We record this fact in full generality as it will be needed in the proofs of Lemmas~\ref{factor1}--\ref{factor2} below. 

\begin{lemma}\label{equiv}
	If $N$ is a 1-connected nilpotent Lie group with its Lie algebra $\gg$, 
	then both the quotient mapping $\ng^*\to\ng^*/N$, $\xi\mapsto N\xi$, and the Kirillov homeomorphism $\kappa\colon\ng^*/N\to\widehat{N}$ are equivariant with respect to the natural actions of the automorphism group $\Aut(N)$. 
\end{lemma}

\begin{proof}
	The equivariance property of the quotient mapping $\ng^*\to\ng^*/N$ is straightforward, while the equivariance of the Kirillov correspondence $\kappa$ is based on its definition in terms of induced representations along with the equivariance property of the construction of induced representations, 
	cf., e.g., \cite[Lemma 2.1.3]{CG90}. 
\end{proof}

\begin{remark}\label{ell_constr}
	\normalfont
	For later use, we recall the construction of the Puk\'anszky correspondence $\ell\colon\RelS\to\stackrel{\frown}{G}_\nor$. 
	For arbitrary $\O\in\RelS$ there exists a unique coadjoint 
	quasi-orbit $\Oc\in(\gg^*/G)^\sim$ with 
	$\O\in(\Bun(\Oc)/G)^\approx$. 
	There exists a connected (not necessarily simply connected) solvable Lie group $\Gg$ that acts transitively on~$\O$ for which $G\subseteq\Gg$ as a closed subgroup with $[\Gg,\Gg]=[G,G]=D$, by \cite[Ch. II, Prop. 7.1, page 539]{Pu71}. 
	Also, there exists a positive $\Gg$-invariant (hence $G$-invariant) Borel measure~$\mu_\O$ on~$\O$, which is unique up to a positive scalar multiple, by \cite[Ch. III, Prop. 1.1, page 545]{Pu71} and its proof. 
	Let $\underline{\hg}:=(\hg_p)_{p\in\O}$ be a $\Gg$-invariant field of polarizations 
	(cf. \cite[Ch. III, \S 2.2, page 548]{Pu71}),
	and for every $p\in\O$  define the semifinite factor unitary representation $T(p):=\ind(\hg_p,p)$ of $G$ as in \cite[Ch. I, Th. 1, page 512]{Pu71}. 
	Then $(T(p))_{p\in\O}$ is a measurable field of representations 
	and if we define 
	\begin{equation}\label{ell_constr_eq1}
		T_{\underline{\hg}}:=\mathop{\int^\oplus}\limits_{\O}T(p)\de\mu_\O(p)
	\end{equation} 
	then $T_{\underline{\hg}}$ is a semifinite factor representation of $G$ with $\ell(\O)=[T_{\underline{\hg}}]^\frown\in\stackrel{\frown}{G}_\nor$. 
	(See \cite[Ch. III, Th. 2, page 551]{Pu71} and \cite[\S 6]{Pu74}.)
\end{remark}

\begin{remark}\label{ampl}
	\normalfont
	We fix a 1-connected solvable Lie group $\Galg$ with its Lie algebra $\ggalg$ for which $G\subseteq \Galg$ is a closed subgroup, $[\gg,\gg]=[\ggalg,\ggalg]=\dg$, and $\ggalg$ is isomorphic to an algebraic Lie algebra. 
		Such a group always exists, cf. \cite[page 521]{Pu71}, where $\Galg$ is denoted by $\widetilde{G}$, or \cite[Subsect. 2.1]{BB24}.
	
The group  $\Galg$ acts on $D$ by 
	automorphisms, since $D=[\Galg,\Galg]$ is a closed normal subgroup of $\Galg$, and this action gives an action of $\Galg$ on $\widehat{D}$.  
As the action of the additive group $\dg^\perp$  on $\gg^*$ by translations commutes with the action of $\Galg$, there  is also a natural group action $(\Galg\times\dg^\perp)\times\gg^*\to\gg^*$, 
$((a,\sigma),\xi)\mapsto a\xi+\sigma$. (See \cite[Rem.~2.3]{BB24}.)

Let $\Omega\in\gg^*/(\Galg+\dg^\perp)$ and $\widetilde{\Omega}:=\iota^*(\Omega)\subseteq\dg^*$. 
	Fix an arbitrary $\xi\in\Omega$. 
	We then have $\Omega=(\Galg+\dg^\perp)\xi$ and $\widetilde{\Omega}=\Galg\xi_0$, where $\xi_0:=\xi\vert_\dg=\iota^*(\xi)$. 
	Moreover, $\Omega$ carries a $(\Galg+\dg^\perp)$-invariant Borel measure $\de\nu$, by \cite[Ch. III, \S 3.1.a, page 552]{Pu71}.
	One can then construct a measure $\de\rho$ on $\Bun(\Omega)$ which is equal to $\de\mu\times\de\varphi$ in every trivialization of the trivial principal bundle $\tau\vert_{\Bun(\Omega)}\colon\Bun(\Omega)\to\Omega$, where $\de\varphi$ denotes the probability Haar measure of the fiber of that bundle, by \cite[Ch. III, \S 3.3.c, page 554]{Pu71}. 
	
	If we denote $K:=\overline{G}(\xi)D$, then $K$ is a closed normal subgroup of $G$ which does not depend on the choice of $\xi\in\Omega$. 
	There exists a Borel field of unitary representations $(U(p))_{p\in\Bun(\Omega)}$ of $K$ such that 
	the representation $\Ind_K^G (U(p))$ is unitary equivalent to the representation $T(p)=\ind(\hg_p,p)$ from Remark~\ref{ell_constr}
	for every  $\O\in(\Bun(\Omega)/G)^\approx$ and every $\Gg$-invariant field of polarizations  
	$\underline{\hg}:=(\hg_p)_{p\in\O}$. 
	See \cite[Ch. III, \S 3.3.d, page~554]{Pu71} and \cite[Ch. III, \S 2.2, page 549]{Pu71}. 
	Let us define 
	$$T:=\mathop{\int^\oplus}\limits_{\Bun(\Omega)}\Ind_K^G (U(p))\de\rho(p).$$ 
	Then there exists a Borel measure $\de\O$ on $(\Bun(\Omega)/G)^\approx$ such that one has a central direct integral decomposition 
	$$T=\mathop{\int^\oplus}\limits_{(\Bun(\Omega)/G)^\approx}T(\O) \de\O, $$
	where $T(\O)$ is unitary equivalent to the representation $T_{\underline{\hg}}$ in \eqref{ell_constr_eq1},  for every $\O\in(\Bun(\Omega)/G)^\approx$, by \cite[Ch. III, Lemma 3.4.4 page 561]{Pu71}.  
	
	We now use the Kirillov correspondence $\kappa\colon \dg^*/D\to\widehat{D}$ for the nilpotent Lie group $D$ to define $[\pi_0]:=\kappa(D\xi_0)\in\widehat{D}$. 
	We also select a Borel measurable field of unitary irreducible representations $(\pi(\zeta))_{\zeta\in\widehat{D}}$ with $[\pi(\zeta)]=\zeta$ for every $\zeta\in\widehat{D}$. 
	(See \cite[Lemma 8.6.2]{Di64}.) 
	Then the representation $T$ is unitary equivalent to 
	$$m\cdot\Ind_D^G\Bigl(\mathop{\int^\oplus}\limits_{\Galg[\pi_0]}\pi(\zeta)\de\nu(\zeta)\Bigr), $$
	where $\de\nu$ is a $\Galg$-invariant measure on $\Galg[\pi_0]$. 
	Here $m=\aleph_0$ if the group $D$ is noncommutative, and $m=1$ if the group $D$ is commutative. 
	See \cite[Ch. III, Lemma 3.4.5, page 563]{Pu71} and its proof.  
	In particular, we obtain the unitary equivalence of unitary representations of $G$
	\begin{equation}
	\label{ampl_eq1}
	\mathop{\int^\oplus}\limits_{(\Bun(\Omega)/G)^\approx}T(\O) \de\O
	\simeq
	m\cdot\Ind_D^G\Bigl(\mathop{\int^\oplus}\limits_{\Galg[\pi_0]}\pi(\zeta)\de\nu(\zeta)\Bigr)
	\end{equation}
which is needed is the proof of Lemma~\ref{factor1} below. 
\end{remark}

\section{Solvable Lie groups having factor regular representations}
\label{Sect3}

In this section we obtain our main result on the characterization of solvable Lie groups having factor regular representations (Theorem~\ref{factor3}). 
The von Neumann algebras generated by the regular representations of these groups provide geometric realizations for the standard representation of the hyperfinite $\II_\infty$ factor (Corollary~\ref{typeI-fact}). 
We also give 
necessary conditions on the Lie algebras of 1-connected solvable Lie groups to have 
factor regular representations, 
in terms of their universal enveloping algebras (Corollary~\ref{typeI-fact2}). 
In particular, we show that every Casimir function is necessarily constant.

We start with a general technical lemma. 

\begin{lemma}\label{hp}
	Let $N$ be separable, unimodular, type~$\I$, locally compact group with a fixed Haar measure $\de\lambda(x)$ and its corresponding Plancherel measure $\de\widehat{\lambda}([\pi])$ on~$\widehat{N}$. 
	Assume that  $\Gamma$ is a topological group with a group homomorphism  $\alpha\colon \Gamma\to\Aut(N)$, $\gamma\mapsto\alpha_\gamma$ such that the mapping $\Gamma\times N\to N$, $(\gamma,x)\mapsto\alpha_\gamma(x)$, is continuous.
	Then
	for every $\gamma\in\Gamma$ there exists $\vert\alpha_\gamma\vert\in (0,\infty)$ satisfying $(\alpha_\gamma)_*(\de \lambda)=\vert\alpha_\gamma\vert^{-1}\de \lambda$ and $(\alpha_\gamma)_*(\de\widehat{\lambda})=\vert\alpha_\gamma\vert^{-1} \de\widehat{\lambda}$.
	Moreover,  the mapping $\Gamma\to(0,\infty)$, $\gamma\mapsto\vert\alpha_\gamma\vert$ is a continuous group homomorphism. 
\end{lemma} 

\begin{proof}
	The existence of the continuous group homomorphism $\Gamma\to(0,\infty)$, $\gamma\mapsto\vert\alpha_\gamma\vert$ satisfying $(\alpha_\gamma)_*(\de \lambda)=\vert\alpha_\gamma\vert^{-1}\de \lambda$ for all $\gamma\in\Gamma$ follows by \cite[Ch. VII, \S 1, no. 4, Prop. 4]{Bo07}. 
	Hence for every $\gamma\in\Gamma$ we have 
	$$\int\limits_G\varphi\,\de (\alpha_\gamma)_*(\lambda)
	:=\int\limits_G\varphi\circ\alpha_\gamma\de \lambda
	=\vert\alpha_\gamma\vert^{-1} \int\limits_G\varphi\, \de \lambda$$ 
	for every $\varphi\in L^1(N)$. 
	On the other hand, the Plancherel measure $\de\widehat{\lambda}$ is the u\-nique measure on $\widehat{N}$ 
	satisfying 
	$$\int\limits_N\vert\varphi(x)\vert^2\de \lambda(x)
	=\int\limits_{\widehat{N}}\Tr(\pi (\varphi)\pi (\varphi)^*)\de\widehat{\lambda}[\pi]$$ 
	for all $\varphi\in L^1(N)\cap L^2(N)$,
	with $\pi(\varphi)=\int\limits_N\varphi(x)\pi (x)\de\lambda(x)$. 
	(See \cite[Th.~18.8.2]{Di64}.)
	We then have 
	\allowdisplaybreaks
	\begin{align*}
		\pi (\varphi\circ\alpha_\gamma)
		& =\int\limits_N\varphi(x)(\pi\circ\alpha_\gamma^{-1})(x)\de(\alpha_\gamma)_*(\lambda)(x) 
		=\vert \alpha_\gamma\vert^{-1} \int\limits_N\varphi(x)(\pi\circ\alpha_\gamma^{-1})(x)\de\lambda(x) \\
		&=\vert \alpha_\gamma \vert^{-1} (\pi\circ\alpha_\gamma^{-1})(\varphi)
	\end{align*}
	hence 
	\allowdisplaybreaks
	\begin{align*}
		\vert\alpha_\gamma\vert^{-1} \int\limits_N\vert\varphi(x)\vert^2\de \lambda(x)
		&= \int\limits_N\vert(\varphi\circ\alpha_\gamma)(x)\vert^2\de \lambda(x) \\
		&= \int\limits_{\widehat{N}}\Tr(\pi (\varphi\circ\alpha_\gamma)\pi (\varphi\circ\alpha_\gamma)^*)\de\widehat{\lambda}[\pi] \\
		&=\vert\alpha_\gamma \vert^{-2 }
		\int\limits_{\widehat{N}}\Tr((\pi\circ\alpha_\gamma^{-1})(\varphi)(\pi\circ\alpha_\gamma^{-1})(\varphi)^*)\de\widehat{\lambda}[\pi] \\
		&=\vert\alpha_\gamma\vert^{-2} 
		\int\limits_{\widehat{N}}\Tr(\pi (\varphi)\pi (\varphi)^*)
		\de(\alpha_\gamma)_*^{-1}(\widehat{\lambda})[\pi].
	\end{align*}
	The uniqueness property  of the Plancherel measure then implies 
	$(\alpha_\gamma)_*^{-1}(\widehat{\lambda})=\vert\alpha_\gamma\vert \widehat{\lambda}$, 
	and this completes the proof.
\end{proof}

We resume our setting where $G$ is a 1-connected solvable Lie group, $D=[G, G]$, and $\Galg$ is a 1-connected solvable Lie group 
with its Lie algebra $\ggalg$ isomorphic with an algebraic Lie algebra, 
$G \subseteq \Galg$ as a closed subgroup, and $[\gg, \gg] =[\ggalg, 	\ggalg]$, 
as in Remark~\ref{ampl}.

\begin{lemma}\label{hp-2}
	Let $\de\lambda$ be a Haar measure on $D$ with its corresponding Plancherel measure $\de\widehat{\lambda}$ on 
	$\widehat{D}$. 
	Assume that  for $\pi_0\in \widehat{D}$, the $\Galg$-orbit
	$\Galg[\pi_0]$  is a non-empty open subset of $\widehat{D}$.
	Then any  non-zero $\Galg$-invariant measure $\de\nu$ on $\Galg[\pi_0]$ 
	is equivalent to the restriction of 
	$\de\widehat{\lambda}$ to $\Galg[\pi_0]$. 
\end{lemma}

\begin{proof}
	Since $D=[\Galg,\Galg]$ is a closed normal subgroup of $\Galg$, the group $\Galg$ acts on $D$ by 
	automorphisms defined as restrictions of the inner automorphisms of $\Galg$. 
	Then, by Lemma~\ref{hp}, the restriction of $\de\widehat{\lambda}$ to $\Galg[\pi_0]$ is $\Galg$- quasi-invariant.  
	Since $\Galg[\pi_0]\subseteq\widehat{D}$ is a nonempty open subset and the support of the Plancherel measure  $\de\widehat{\lambda}$ is equal to $\widehat{D}$ by \cite[18.8.4]{Di64}, it follows that the Plancherel measure of $\Galg[\pi_0]$ is different from~$0$, hence the restriction of $\de\widehat{\lambda}$ to $\Galg[\pi_0]$ is 
	a nonzero 
	quasi-invariant measure. 
	The statement is then a consequence of the uniqueness, up to equivalence, of 
 quasi-invariant measures on a homogeneous space (\cite[Ch. VII, \S 2, no. 5, Th. 1]{Bo07}).
\end{proof}

\begin{lemma}\label{factor1}
	Let $G$ be a 1-connected solvable Lie group. 
	If there exists a coadjoint quasi-orbit $\Oc\in(\gg^*/G)^\sim$ such that $\Oc$ is an open dense subset of $\gg^*$ and $\Bun(\Oc)\in(\Bun(\Oc)/G)^\approx$, then the left regular representation $\lambda_G\colon G\to\Bc(L^2(G))$ is a factor representation. 
\end{lemma}

\begin{proof}
	Select any $\xi\in\Oc$ and denote $\xi_0:=\xi\vert_\dg\in\dg^*$. 
	The restriction mapping $\iota^*\colon\gg^*\to\dg^*$  
	is open and surjective, therefore
	$\iota^*(\Oc)$ is an open dense subset of $\dg^*$. 
	On the other hand, since $[\ggalg,\ggalg]=[\gg,\gg]=\dg$, 
	we have $\Galg\dg\subseteq\dg$, and this implies that the map $\iota^*$ is $\Galg$-equivariant. 
	  By \cite[Lemma 3.4]{BB24}, 
	\begin{equation}
	\label{factor1_proof_eq1}
	\Oc=\Galg\xi=(\Galg+\dg^\perp)\xi,
	\end{equation}
and 
	$\iota^*(\Oc)=\Galg\xi_0$. 
	Thus $\Galg\xi_0\subseteq\dg^*$ is an open dense subset. 
	Since $D\subseteq\Galg$, we also have $D\Galg\xi_0=\Galg\xi_0$. 
	Then, since the quotient mapping $\dg^*\to\dg^*/D$ is an open surjective map, 
	it follows that $(\Galg\xi_0)/D\subseteq \dg^*/D$ is an open dense subset. 
	Taking into account the Kirillov homeomorphism $\kappa\colon \dg^*/D\to\widehat{D}$ and the fact that the support of the Plancherel measure of the nilpotent Lie group $D$ is equal to $\widehat{D}$, 
	it follows that the the complement of the subset $\kappa((\Galg\xi_0)/D)\subseteq\widehat{D}$ is negligible with respect to the Plancherel measure. 
	The homeomorphism $\kappa$ is $\Galg$-equivariant 
	by Lemma~\ref{equiv}, hence we get that  $\kappa((\Galg\xi_0)/D)=\Galg[\pi_0]$, where $[\pi_0]=\kappa(D\xi_0)\in\widehat{D}$. 
	
	We now select a Borel measurable field of unitary irreducible representations $(\pi(\zeta))_{\zeta\in\widehat{D}}$ with $[\pi(\zeta)]=\zeta$ for every $\zeta\in\widehat{D}$. 
	(See \cite[Lemma 8.6.2]{Di64}.) 
	Denoting by $\de\zeta$ the Plancherel measure on $\widehat{D}$ corresponding to a fixed Haar measure on $D$, we have 
	\begin{equation*}
		\lambda_D=m\cdot \mathop{\int^\oplus}\limits_{\widehat{D}}\pi(\zeta)\de\zeta
		=m\cdot\mathop{\int^\oplus}\limits_{\Galg[\pi_0]}\pi(\zeta)\de\zeta
	\end{equation*}
	where $m=\aleph_0$ if the group $D$ is noncommutative, and $m=1$ if the group $D$ is commutative. 
	For a locally compact group, the left regular representation can be obtained by inducing the trivial representation of the trivial subgroup; thus induction in stages gives $\lambda_G=\Ind_D^G(\lambda_D)$.
	We then obtain by \cite[\S 10, Thm. 10.1]{Ma52}, 
	\begin{equation}\label{factor1*}
		\lambda_G
		=m\cdot\mathop{\int^\oplus}\limits_{\Galg[\pi_0]}\Ind_D^G(\pi(\zeta))\de\zeta. 
	\end{equation}
	Since $\O:=\Bun(\Oc) \in(\Bun(\Oc)/G)^\approx$, we may apply the construction outlined in Remark~\ref{ell_constr} to obtain the semifinite factor representation 
	$T_{\underline{\hg}}:=\mathop{\int^\oplus}\limits_{\O}T(p)\de\mu_\O(p)$
	of $G$ with $\ell(\O)=[T_{\underline{\hg}}]^\frown\in\stackrel{\frown}{G}_\nor$. 
	We also note that, by \eqref{factor1_proof_eq1}, we have $\Oc\in\gg^*/(\Galg+\dg^\perp)$ hence we may apply Remark~\ref{ampl} with $\Omega:=\Oc$ to obtain the $\Galg$-invariant measure $\de\nu$ on $\Galg[\pi_0]$ for which \eqref{ampl_eq1} holds. 
	Then we have
	\begin{align*}
		\lambda_G & = m \cdot \mathop{\int^\oplus}\limits_{\Galg[\pi_0]} \Ind_D^G (\pi(\zeta) )\de \zeta 
		\stackrel {(*)}{\simeq} m \cdot\mathop{\int^\oplus}\limits_{\Galg[\pi_0]} \Ind_D^G (\pi(\zeta)) \de \nu (\zeta)
		\\
		& \stackrel {
			\eqref{ampl_eq1}
		}{\simeq} m^2 \cdot\mathop{\int^\oplus}\limits_{(\Bun (\Oc)/G)^\approx}  T(O)   \de O  
		\stackrel {(**) }{=}  m^2 \cdot T(\O) 
		\simeq m^2 \cdot T_{\hg}, 
	\end{align*}
	where the unitary equivalence $(*)$ is a consequence of Lemma~\ref{hp-2}, 
	while~$(**)$ 
	follows by the hypothesis $(\Bun (\Oc)/G)^\approx=\{\O\}$.
	Therefore $\lambda_G$ shares the properties of $T_{\underline{\hg}}$ of being a semifinite factor representation of~$G$, 
	and this completes the proof. 
\end{proof}

Lemma~\ref{factor2} below provides a converse to Lemma~\ref{factor1}. 

\begin{lemma}\label{factor2}
	Let $G$ be a 1-connected solvable Lie group. 
	If the left regular representation $\lambda_G\colon G\to\Bc(L^2(G))$ is a factor representation, then its quasi-equivalence class $[\lambda_G]^\frown\in\stackrel{\frown}{G}_\nor$ 
	is square integrable. 
	Furthermore  $[\lambda_G]^\frown=\ell(\O)$, 
	where  $\O=\Bun(\Oc)\in(\Bun(\Oc)/G)^\approx$ and $\Oc\in(\gg^*/G)^\sim$ is a coadjoint quasi-orbit which is an open dense subset of $\gg^*$. 
\end{lemma}

\begin{proof}
	Step 1: Since $\lambda_G\colon G\to\Bc(L^2(G))$ is a factor representation,  
	it follows by \cite[Prop. 2.3]{Ros78} that $\lambda_G$ is a square-integrable representation. 
	By 
	\cite[Th. 3.1]{BB24}
	we then obtain an open coadjoint quasi-orbit $\Oc\in(\gg^*/G)^\sim$ for which $\O:=\Bun(\Oc)\in(\Bun(\Oc)/G)^\approx$ and $\ell(\O)=[\lambda_G]^\frown\in\stackrel{\frown}{G}_\nor$. 
	An application of 
	\cite[Lemma 3.4]{BB24} 
	for $\Omega:=\Oc$ shows that 
	$\iota^*(\Oc)\in\dg^*/\Galg$, 
	and moreover $\iota^*(\Oc)\subseteq\dg^*$ 
	is an open subset and 
	$\Oc=(\iota^*)^{-1}(\iota^*(\Oc))$. Here we recall that 
	$\iota^*\colon\gg^*\to\dg^*$ is the restriction map.
	It remains to prove that the subset $\Oc\subseteq\gg^*$ is dense. 
	
	Step 2: 
	We use again a Borel measurable field of unitary irreducible representations $(\pi(\zeta))_{\zeta\in\widehat{D}}$ with $[\pi(\zeta)]=\zeta$ for every $\zeta\in\widehat{D}$. 
	As shown in \cite[Ch. V, Lemma 9.1, page 603]{Pu71},  there exist a measure $\de O$ on the countably separated Borel space $\widehat{D}/\Galg$ and a $\Galg$-quasi-invariant measure $\nu_O$ on every $O\in\widehat{D}/\Galg$ satisfying the following conditions: 
	\begin{enumerate}[{\rm(a)}]
		\item\label{factor2_proof_item-a} 
		The Plancherel measure $\de\zeta$ decomposes as the continuous direct sum of the family $(\nu_O)_{O\in \widehat{D}/\Galg}$ with respect to the measure~$\de O$, in the sense of \cite[\S 11]{Ma52}. 
		\item\label{factor2_proof_item-b} 
		The direct integral 
		\begin{equation}\label{factor2_proof_eq1}
			M:=\mathop{\int^\oplus}\limits_{\widehat{D}/\Galg}T(O)\de O
		\end{equation}
		is a central decomposition of unitary representations of $G$, where we use the representations 
		$$T(O):=\Ind_D^G (U(O))\; \text{ with }\;
		U(O):=\mathop{\int^\oplus}\limits_O \pi(\zeta)\de\nu_O(\zeta).$$ 
		\item\label{factor2_proof_item-c} 
		We have $\lambda_G= m\cdot M$, where $m= \aleph_0$ if $D$ is not commutative, and $m=1$ otherwise.
		\end{enumerate}
	Using the hypothesis that $\lambda_G$ is a factor representation 
	along with~\eqref{factor2_proof_item-c}, it follows that the unitary representation~$M$ defined in~\eqref{factor2_proof_eq1} is a factor representation as well. 
	The condition that \eqref{factor2_proof_eq1} is a central decomposition means that $L^\infty(\widehat{D}/\Galg,\de O)$ embeds (via multiplication operators) into the centre of the von Neumann algebra generated by~$M$. 
	Since $M$ is  factor representation, we then obtain $\dim L^\infty(\widehat{D}/\Galg,\de O)=1$, 
	and this implies that the measure $\de O$ takes only the values $0$ and $1$. 
	We have already noted above that the Borel space $\widehat{D}/\Galg$ is countably separated, 
	and then every $\{0,1\}$-valued Borel measure on $\widehat{D}/\Galg$ is a point mass. 
	(See \cite[Sect. 3.4, Lemma, page 78]{Arv76}.) 
	Thus the measure $\de O$ is a point mass, which implies by~\eqref{factor2_proof_eq1} that there exists a unique $\Galg$-orbit $O_0=\Galg[\pi_0]\in \widehat{D}/\Galg$ with 
	$$M=T(O_0)=\Ind_D^G (U(O_0))=\Ind_D^G\Bigl(\mathop{\int^\oplus}\limits_{\Galg[\pi_0]} \pi(\zeta)\de\nu_{O_0}(\zeta)\Bigr).$$ 
	On the other hand, it follows by~\eqref{factor2_proof_item-a} that the Plancherel masure of~$\widehat{D}$ is supported by $\Galg[\pi_0]$. 
	This implies by \cite[Th., page 286]{Gr80} that 
	$\Galg[\pi_0]\subseteq \widehat{D}$ is an open subset. 
	Moreover, since the support of the Plancherel measure is equal to $\widehat{D}$ by \cite[18.8.4]{Di64}, it also follows that 
	$\Galg[\pi_0]$ is dense in $\widehat{D}$. 
	
	Step 3: 
	For arbitrary $\xi\in\dg^*$, let us denote by $q(\xi)\in\widehat{D}$ the point of the unitary dual that corresponds to the coadjoint $D$-orbit $D\xi\in\dg^*/D$ via the Kirillov homeomorphism $\kappa\colon \dg^*/D\mathop{\to}\limits^\sim\widehat{D}$. 
	The mapping  $q\colon\dg^*\to\widehat{D}$ defined in this way is continuous, open, and surjective, and it is also $\Aut(D)$-equivariant, in particular $\Galg$-equivariant by Lemma~\ref{equiv}.
	For arbitrary $\xi\in\dg^*$ we then obtain 
	\begin{align*}
		q^{-1}(\Galg q(\xi))
		&=\{\eta\in\dg^*\mid (\exists x\in\Galg)\quad q(\eta)= xq(\xi)\} \\
		&=\{\eta\in\dg^*\mid (\exists x\in\Galg)\quad q(\eta)= q(x\xi)\} \\
		&=\{\eta\in\dg^*\mid (\exists x\in\Galg)(\exists y\in D) \quad \eta= yx\xi\}
	\end{align*}
	hence, since $D\subseteq\Galg$, we have 
	\begin{equation*}
		q^{-1}(\Galg q(\xi))=\Galg\xi\text{ for all }\xi\in\dg^*. 
	\end{equation*}
	Selecting any $\xi_0\in\dg^*$ with $q(\xi_0)=[\pi_0]$, 
	we then obtain  $$q^{-1}(\Galg[\pi_0])=\Galg\xi_0\subseteq\dg^*.$$
	Since $q\colon\dg^*\to\widehat{D}$ is a continuous, open, and surjective mapping, 
	and we have established above that $\Galg[\pi_0]\subseteq \widehat{D}$ is an open dense subset, it follows that
	$\Galg\xi_0\subseteq\dg^*$ is an open dense subset as well. 
	
	Step 4: 
	Since $\Galg\xi_0\subseteq\dg^*$ is a dense subset, 
	while $\iota^*(\Oc)\subseteq\dg^*$ is an open subset, 
	we obtain $\Galg\xi_0\cap \iota^*(\Oc)\ne\emptyset$. 
	Since $\iota^*(\Oc)$ is a $\Galg$-orbit in $\dg^*$ (see Step 1), 
	we then obtain $\Galg\xi_0=\iota^*(\Oc)$. 
	Therefore $\Oc=(\iota^*)^{-1}(\Galg\xi_0)$. 
	Now, since the restriction mapping $\iota^*\colon\gg^*\to\dg^*$ is continuous, open, surjective, and $\Galg$-equivariant, while $\Galg\xi_0\subseteq\dg^*$ is an open dense subset, 
	the open subset $\Oc\subseteq\gg^*$ is  dense, as well.
	This completes the proof. 
\end{proof}

\begin{theorem}
	\label{factor3}
	If $G$ is a 1-connected solvable Lie group, then its left regular representation $\lambda_G\colon G\to\Bc(L^2(G))$ is a factor representation
	if and only if there exists a coadjoint quasi-orbit $\Oc\in(\gg^*/G)^\sim$  which is an open dense subset  of $\gg^*$ and satisfies $\Bun(\Oc)\in(\Bun(\Oc)/G)^\approx$. 
	If this is the case, then the quasi-equivalence class $[\lambda_G]^\frown\in\stackrel{\frown}{G}_\nor$ 
	is square integrable,  we have $[\lambda_G]^\frown=\ell(\O)$ 
	for  $\O:=\Bun(\Oc)$, and moreover $\Oc$ is the only open coadjoint quasi-orbit of~$G$. 
\end{theorem}

\begin{proof}
	The assertions follow by Lemmas \ref{factor1} and \ref{factor2}. 
	The uniqueness property of $\Oc$ follows from the fact that $\Oc\subseteq\gg^*$ is dense, hence $\Oc$ has a nonempty intersection with every open subset of $\gg^*$, while the coadjoint quasi-orbits are mutually disjoint, hence no coadjoint quasi-orbit different from $\Oc$ could be an open subset of~$\gg^*$. 
\end{proof}

\begin{corollary}\label{typeI-fact}
	If the regular representation of a 1-connected solvable Lie group is a factor representation, then  it is standard hyperfinite type~$\II_\infty$.
\end{corollary}

\begin{proof}
	Assume that $G$ is a solvable Lie group such that its left regular representation
	$\lambda_G\colon  G \to \Bc(L^2(G))$  is a factor representation.
	First note that in this case  $G$ cannot be abelian, so we let $G$ be a non-abelian solvable Lie group.
	
	Assume  in addition that 
	$\lambda_G$  is  type~$\I$. 
	Then by Theorem~\ref{factor3} and  
	\cite[Cor. 5.6]{BB24}, 
	$[\lambda_G]^\sim =\ell(\Oc)$ where 
	$\Oc$ is a unique simply connected open and dense coadjoint orbit.
	However, 
    by \cite[Prop. 2.6]{BB24}, 
	the number of simply connected open coadjoint orbits is even. 
	We have thus obtained a contradiction, hence $\lambda_G$ must be of type $\II$ or $\III$. 
	
	On the other hand,  $\lambda_G$ cannot be of type $\III$ (see \cite[Thm. 5, \S 9]{Pu71}). 
	Furthermore, since $G$ is not abelian,  there is no finite faithful  trace on  $C^*(G)$ 
	(see \cite[Cor.~2.9]{BB21}), hence 
	$\lambda_G$ is not type $\II_1$. 
	Finally, the von Neumann algebra $\lambda_G(G)''$ is hyperfinite by \cite[Cor.7]{C76} and is standard by \cite[Cor. 3.5.6]{ES92}.
\end{proof}

\begin{corollary}\label{typeI-fact2}
	Let $G$ be a 1-connected solvable Lie group with its Lie algebra~$\gg$. 
	If the regular representation of $G$ is a factor representation, then the universal associative enveloping algebra $U(\gg)$ is primitive 
	and the center of the division ring of $U(\gg)$ is equal to $\RR\1$. 
	Moreover, every Casimir function on the Lie-Poisson space $\gg^*$ is constant.
\end{corollary}

\begin{proof}
The regular representation of $G$ is a factor representation hence, by Theorem~\ref{factor3} and \cite[Th. 3.1]{BB24}, the centre of $G$ is trivial, 
and then the centre of~$\gg$ is trivial, too. 
Then the 
adjoint representation 
 $\ad_\gg\colon\gg\to\Der(\gg)$ is faithful. 
In particular, the linear Lie algebra $\gg_0:=\ad_\gg(\gg)\subseteq\Der(\gg)$ is isomorphic to~$\gg$, and we denote by $G_0$ the 1-connected, solvable Lie group whose Lie algebra is~$\gg_0$, hence $G_0$ and $G$ are isomorphic Lie groups. 

Let us denote by $\ggalg\subseteq\Der(\gg)$ the algebraic hull of~$\gg_0$.  
Then the auxiliary group~$\Galg$ in 
Remark~\ref{ampl} 
may be taken as the 1-connected Lie group whose Lie algebra is~$\ggalg$. 
If we denote by $\Galg_0\subseteq\Aut(\gg)$ the integral (connected) subgroup of $\Aut(\gg)$ whose Lie algebra is $\ggalg$, then we have the universal covering homomorphism $q\colon \Galg\to\Galg_0$, satisfying
$\Ad_{\Galg_0}\circ q=\Ad_{\Galg}\colon \Galg\to\Aut(\ggalg)$. 
Since $\gg_0$ is an ideal of $\ggalg$, 
we then obtain coadjoint actions of $\Galg_0$ and $\Galg$ on $\gg_0^*$ 
with 
\begin{equation}
\label{typeI-fact2_proof_eq1}
(\forall \xi\in\gg_0^*)\quad \Galg\xi=\Galg_0\xi
\end{equation}
The Lie algebras $\gg$ and $\gg_0$ are isomorphic, 
hence, by Theorem~\ref{factor3}, there exists $\Oc\in(\gg_0/G_0)^*$ 
which is an open dense subset of $\gg_0^*$. 
Then, by \cite[Eq. (2.1)--(2.2)]{BB24}, we have $\widetilde{G_0\xi}=\gg_0^*$ and $\Oc=\Galg\xi$ if $\xi\in\Oc$. 
Then, by \eqref{typeI-fact2_proof_eq1}, $\Oc=\Galg_0\xi$ is an open dense orbit of $\Galg_0$ in~$\gg_0^*$,  
and the assertions on the universal associative enveloping algebra $U(\gg)$ follow by \cite[Thm.]{Oo76}. 

Finally, we recall from the beginning of the Introduction that a function $c\in\Ci(\gg^*)$ is a Casimir function if and only if $c(g\xi)=c(\xi)$ for all $\xi\in\gg^*$ and $g\in G$. 
Since the function $c\colon\gg^*\to\RR$ is in particular continuous, it follows that for every $\xi\in \gg^*$ the restriction of $c$ to $\overline{G\xi}$ is constant. 
Recalling that the coadjoint quasi-orbit of every $\xi\in\gg^*$ is contained in $\overline{G\xi}$, it the follows that $c$ is constant on every coadjoint quasi-orbit. 
By  Theorem~\ref{factor3}, there exists a coadjoint quasi-orbit which is dense in $\gg^*$, 
hence we obtain that the function $c\colon\gg^*\to\RR$ is constant. 
\end{proof}

\begin{lemma}\label{fact-prim}
Let $G$ be an amenable, separable, locally compact group with its regular representation $\lambda_G$. 
If $\lambda_G$ is a factor representation, then the group $C^*$-algebra of $G$ is primitive. 
\end{lemma}

\begin{proof}
Let $\lambda_G\colon C^*(G)\to\Bc(L^2(G))$ be the faithful $*$-representation generated by the regular representation. 
Since $G$ is separable, it follows that $C^*(G)$ is separable, 
by \cite[13.9.2]{Di64}. 
Then, by 
\cite[Th. 2, Cor. 3]{Di60},  the kernel of the factor representation $\lambda_G\colon C^*(G)\to\Bc(L^2(G))$ is a primitive ideal of $C^*(G)$.
Thus $\{0\}$ is a primitive ideal of $C^*(G)$ and this completes the proof. 
 \end{proof}

\begin{corollary}\label{primitive}
Let $G$  be a 1-connected solvable Lie group such that  there exists a coadjoint quasi-orbit $\Oc\in(\gg^*/G)^\sim$  which is an open dense subset  of $\gg^*$ and satisfies $\Bun(\Oc)\in(\Bun(\Oc)/G)^\approx$. 
Then its $C^*$-algebra $C^*(G)$ is primitive.
\end{corollary}

\begin{proof}
This is a direct consequence of Theorem~\ref{factor3} and Lemma~\ref{fact-prim}.
\end{proof}

\begin{example}\label{many}
\normalfont 
We constructed  in \cite[Ex. 7.2--7.3]{BB24} 
a family of semidirect products of abelian Lie groups $G=\Vc\rtimes A$, 
where $\dim A$ is an arbitrary integer~$\ge3$ while $\Vc$ is the nilradical of $G$, 
and each one of these groups $G$ has  a  coadjoint quasi-orbit $\Oc\in(\gg^*/G)^\sim$ with the  properties: 
\begin{enumerate}[{\rm(i)}]
	\item\label{item_i} $\Oc\subseteq\gg^*$ is a dense open subset; 
	\item\label{item_ii} $\Oc\not\in\gg^*/G$; 
	\item\label{item_iii}  if $(p,\xi)\in\Oc\subseteq\Vc^*\times\ag^*=\gg^*$, then the coadjoint isotropy group $G(p,\xi)$ is abelian and connected. 
\end{enumerate}
It follows by Theorem~\ref{factor3} that the regular representations of every solvable Lie group $G$ with these properties is a factor representation. 
\end{example}

\begin{example}\label{counterexamples}
\normalfont
Let $\KK\in\{\RR,\CC\}$. The semidirect product group 
$$\Gamma_\KK:=\KK\rtimes\KK^\times$$
with the product $(b_1,a_1)\cdot(b_2,a_2)=(b_1+a_1b_2,a_1a_2)$ for all $b_1,b_1\in\KK$, $a_1,a_2\in\KK^\times$, is a solvable Lie group of type~\I, whose unitary dual space 
consists of an open point $[\pi_0]$ and the set of 1-dimensional unitary representations of $\Gamma_\KK$. 
The regular representation $\lambda_{\Gamma_\KK}$ is unitary equivalent to a countably infinite multiple of the unitary irreducible representation $\pi_0$, hence the group von Neumann algebra $\lambda_{\Gamma_\KK}(\Gamma_\KK)''\subseteq\Bc(L^2(\Gamma_\KK))$ is a type~\I\ factor. 
See \cite[\S 3, Prop. 4]{DuMoo76}, \cite[Ex. 6.4]{BB18}, and the references therein. 

If $\KK=\RR$, the group $\Gamma_\RR$ is simply connected, but not connected. 
Its connected $\1$-component $(\Gamma_\RR)_\1=\RR\rtimes\RR_+^\times$ 
is a 1-connected solvable Lie group whose unitary dual space has two open points. 
If $\Gamma=\CC$, the group $\Gamma_\CC$ is connected, but not simply connected. 
Its universal covering group $\widetilde{\Gamma_\CC}$ is a 1-connected solvable Lie group discussed in \cite[Ex. 7.1]{BB24}, and its unitary dual space 
has no open points. 
Summarizing, none of the solvable Lie groups $\Gamma_\RR,\Gamma_\CC,(\Gamma_\RR)_\1,\widetilde{\Gamma_\CC}$ satisfies the hypotheses of Theorem~\ref{factor3}, although all of them have open coadjoint orbits, and for the first two of them the regular representation is a factor representation. 
\end{example}

\begin{remark}
\normalfont
As already mentioned in Introduction, there is no available characterization so far of the locally compact groups whose regular representation is a factor representation.
For any countable 
discrete group, its regular representation is a factor representation if and only if the group under consideration is an ICC (infinite conjugacy classes) group, that is,  the conjugacy class of every element different from the unit element is infinite, 
by 
\cite[Lemma 5.3.4]{MvN43}.

In the case of countable discrete groups satisfying the above ICC condition, 
the factor generated by the regular representations is always type~\II$_1$, 
and a great variety of mutually non-isomorphic factors arise in this way. 
In the case of solvable Lie groups, if the regular representation is a factor representation then the corresponding factor is always isomorphic to the hyperfinite type~\II$_\infty$ factor (Corollary~\ref{typeI-fact}). 
While Example~\ref{many} shows many specific solvable Lie groups with that property, 
the Lie groups with factor regular representation  are not classified yet. 
\end{remark}

\section{Regular representation of solvable Lie groups with open coadjoint orbits}
\label{Sect4}

In this final section we study the regular representation of solvable Lie groups that have open coadjoint orbits. 
We also include some relevant examples. 

Our  characterization of the 1-connected solvable Lie groups with factor regular representation involves the presence of an open coadjoint quasi-orbit which is not a coadjoint orbit. 
	Therefore, it is natural to study the von Neumann algebra generated by the regular representation in the complementary case of solvable Lie groups that have open coadjoint orbits. 
	See \cite{FO22} and \cite{FV21} for applications of these groups in the theory of frames. 
	In this case we find that the corresponding von Neumann algebra 
	is always type \I\ (Corollary~\ref{regI-cor}), but we prove by example that a 
	Frobenius Lie group may not be type~$\I$ (Example~\ref{exF}).

\begin{proposition}
	\label{regI}
	Let $G$ be a 1-connected solvable Lie group 
	and denote 
	$$E:=\{\xi\in\gg^*\mid [G(\xi):\overline{G}(\xi)]<\infty\text{ and }G\xi\text{ is locally closed in }\gg^*\}.$$
	If the Lebesgue measure of the set $\gg^*\setminus E$ is zero, then the von Neumann algebra generated by the left regular representation~$\lambda_G$  is type~$\I$. 
\end{proposition}

\begin{proof}
	As in the proof of Lemmas \ref{factor1}--\ref{factor2},
	 we select a Borel measurable field of unitary irreducible representations $(\pi(\zeta))_{\zeta\in\widehat{D}}$ with $[\pi(\zeta)]=\zeta$ for every $\zeta\in\widehat{D}$ 
	and we obtain $\lambda_G=\Ind_D^G(\lambda_D)$ hence 
	\begin{equation}
		\label{regI_proof_eq1}
		\lambda_G
		=m\cdot \mathop{\int^\oplus}\limits_{\widehat{D}}\Ind_D^G(\pi(\zeta))\de\zeta 
		=m\cdot\mathop{\int^\oplus}\limits_{\widehat{D}/\Galg}
		\Bigl(\mathop{\int^\oplus}\limits_{O}\Ind_D^G(\pi(\zeta))\de\nu_O(\zeta)\Bigr)
		\de O
	\end{equation}
	where $\de\zeta$ is the Plancherel measure on $\widehat{D}$ corresponding to a fixed Haar measure on $D$, 
	a suitable measure $\de O$ on the countably separated Borel space $\widehat{D}/\Galg$ and a suitable $\Galg$-quasi-invariant measure $\nu_O$ on every $O\in\widehat{D}/\Galg$.
	Here	$m=\aleph_0$ if the group $D$ is noncommutative, and $m=1$ if $D$ is commutative. 
	Moreover, $\Galg$ is a solvable Lie group for which $G\subseteq \Galg$ is a closed subgroup, $[\gg,\gg]=[\ggalg,\ggalg]$, and $\ggalg$ is isomorphic to an algebraic Lie algebra. 
	See Remark~\ref{ampl}.
	
	On the other hand, by Remark~\ref{ampl} again, 
for every $[\pi_0]\in\widehat{D}$ and $O:=\Galg[\pi_0]\in\widehat{D}/\Galg$, we get the  unitary equivalence of representations of~$G$ 
	\begin{equation}
		\label{regI_proof_eq2}
		m\cdot\Ind_D^G\Bigl(\mathop{\int^\oplus}\limits_{\Galg[\pi_0]}\pi(\zeta)\de\nu(\zeta)\Bigr)
		\simeq 
		\mathop{\int^\oplus}\limits_{(\Bun(\Omega)/G)^\approx}T(\O) \de\O,
	\end{equation}
	where $\xi\in\gg^*$ has the property that $[\pi_0]$ corresponds to the coadjoint $D$-orbit of $\xi\vert_\dg\in\dg^*$ via Kirillov's correspondence for the nilpotent Lie group~$D$, 
	while $\Omega:=(\Galg+\dg^\perp)\xi\in\gg^*/(\Galg+\dg^\perp)$ and $\de\O$ is a suitable measure on the space $(\Bun(\Omega)/G)^\approx$. 
	The set $E$ in the statement is invariant to the action of the group $\Galg+\dg^\perp$ on $\gg^*$, 
	hence if $\xi\in E$ then $\Omega\subseteq E$. 
	Therefore $T(\O)$ is a type-$\I_\infty$ factor representation of $G$ for every $\O\in(\Bun(\Omega)/G)^\approx$ by 
	\cite[Rem. 5.3]{BB24}.
	Since the Lebesgue measure of $\gg^*\setminus E$ is zero, it then follows by \eqref{regI_proof_eq1}--\eqref{regI_proof_eq2} 
	that the von Neumann algebra generated by $\lambda_G$ is type~$\I$, 
	by \cite[Lemma 9.3]{Pu71}.
\end{proof}

\begin{corollary}
	\label{regI-cor}
	If $G$ is a solvable Lie group 
	that has open coadjoint orbits, then the von Neumann algebra generated by $\lambda_G$  is type~$\I$. 
	Moreover, every Casimir function on the Lie-Poisson space $\gg^*$ is constant.
\end{corollary}

\begin{proof}
	Since $G$ has open coadjoint orbits, it follows that the union of its open coadjoint orbits is an open dense subset $\Omega\subseteq\gg^*$ 
	with finitely many connected components, i.e., coadjoint orbits (see e.g., \cite[Prop. 4.5(ii)]{BB16}). 
	Then the Lebesgue measure of the set $\gg^*\setminus\Omega$ is equal to zero. 
	Moreover, for every $\xi\in\Omega$, its coadjoint orbit $G\xi$ is open, hence it is a locally closed subset of $\gg^*$. 
	On the other hand, since $G\xi$ is open in $\gg^*$, 
	it follows that $G(\xi)_\1=\{\1\}$, hence 
	the coadjoint isotropy group $G(\xi)$ is abelian by \cite[Eq. (2.3)]{BB24}
	and then $\overline{G}(\xi)=G(\xi)$.
	Thus $\Omega\subseteq E$ and we may apply Proposition~\ref{regI}. 
	
	For the assertion on Casimir functions, we recall from the proof of Corollary~\ref{typeI-fact2} that if a function $c\in\Ci(\gg^*)$ is a Casimir function then for every $\xi\in \gg^*$ the restriction of $c$ to $\overline{G\xi}$ is constant. 
		Let us label the open coadjoint orbits of $G$ as $\Oc_1,\dots,\Oc_k$, 
		hence $\Oc_1\sqcup\cdots\sqcup\Oc_k=\Omega$. 
		For $j=1,\dots,k$, there exists $t_j\in\RR$ with $c(\Oc_j)=\{t_j\}$, 
		and we will consider the finite set $F:=\{t_j\mid 1\le j\le k\}\subseteq \RR$. 
		Since $\Omega$ is dense in $\gg^*$ and $c\colon\gg^*\to\RR$ is continuous, 
		we obtain $c(\gg^*)=c(\overline{\Omega})\subseteq \overline{c(\Omega)}=\overline{F}=F$, 
		hence the set of real numbers $c(\gg^*)\subseteq\RR$ is finite. 
		On the other hand, since $\gg^*$ is a connected topological space and the function $c\colon\gg^*\to\RR$ is continuous, the subset $c(\gg^*)\subseteq\RR$ is necessarily connected. 
		Consequently, the set $c(\gg^*)$ is a singleton, that is, the function $c$ is constant.
\end{proof}

Proposition~\ref{regI}  does not necessarily imply that if a solvable Lie group has open coadjoint orbits, then it is type~$\I$. 
We construct below a 
a class of examples  suggested by \cite[Sect. 7]{DiMa14} and \cite[\S 4.2]{BB23}. 
They illustrate, in addition, the well-known fact  that there exist non-type-$\I$ solvable Lie groups whose corresponding von Neumann is type~$\I$. 
(See e.g., \cite{DuMoo76}.)
Indeed, an example of such a solvable 
Lie group is discussed in \cite[Ch. IV, Rem. 9.1]{Pu71}
however these earlier examples  of groups do not have open coadjoint orbits. 

\begin{example}
	\label{exF}
	\normalfont 
	Let $\hg_{2n+1}=\RR^{2n+1}$ be the Heisenberg algebra with its canonical basis  $e_0,e_1,\dots,e_{2n}$ satisfying  $[e_j,e_{n+j}]=e_0$ for $j=1,\dots,n$. 
	We denote by $I_n\in M_n(\RR)$ the identity matrix 
	and for any $A=(a_{jk})_{1\le j,k\le n}\in M_n(\RR)$ and $c\in\RR$ we define 
	\begin{equation}
		\label{exF_eq1}
		B=\begin{pmatrix}
			c & 0 & 0 \\
			0 & A & 0 \\
			0 & 0 & cI_n-A^\top
		\end{pmatrix}
		\colon\hg_{2n+1}\to\hg_{2n+1}
	\end{equation}
	where $A^\top$ is the transpose of the matrix~$A$. 
	We have 
	$$Be_j=\sum\limits_{k=1}^na_{jk}e_k 
	\text{ and }Be_{n+j}=ce_{n+j}-\sum\limits_{r=1}^na_{rj}e_{n+r}$$
	hence 
	$
	[Be_j,e_{n+j}]+[e_j,Be_{n+j}]=a_{jj}e_0+ce_0-a_{jj}e_0=ce_0=Be_0=B[e_j,e_{n+j}].
	$
	Moreover, since $B$ is given by a block-diagonal matrix, it is clear that 
	$[Be_j,e_k]+[e_j,Be_k]=0=B[e_j,e_k]$ if either $1\le j,k\le n$ or $n+1\le j,k\le 2n$. 
	Consequently $B\in\Der(\hg_{2n+1})$ and then we can define the semidirect product 
	$$\gg:=\hg_{2n+1}\rtimes\RR B$$
	and let $G$ be its corresponding 1-connected Lie group. 
	If $c\ne0$, then  $G$ has two open coadjoint orbits by \cite[Lemma 3.7]{BB23}. 
	
	Let $S_A$ be the subgroup
	of $(\RR,+)$ generated by the imaginary parts of the purely imaginary eigenvalues of~$A$. 
	We claim that if $S_A$ is not closed in $\RR$, then the solvable Lie group $G$ is not type~$\I$. 
	In fact, let us note that $\pg:=\spa(\{e_0\}\cup\{e_{n+1},\dots,e_{2n}\})$ is an abelian ideal of 
	$\hg_{2n+1}$ which is invariant under the derivation~$B$, hence $\pg$ is actually an ideal of $\gg$. 
	We then have the short exact sequence of Lie algebras 
	$$0\to\pg\hookrightarrow\gg\to\qg\to0$$
	where $\qg:=\gg/\pg$. 
	Since $\gg$ is a solvable Lie algebra, we obtain the corresponding short exact sequence of Lie groups 
	$$\1\to P\hookrightarrow G\to Q\to\1.$$
	In more detail, $P$ is defined as the integral subgroup of $G$ corresponding to the Lie subalgebra $\pg\subseteq\gg$ and then, since $G$ is a 1-connected solvable Lie group, it follows that $P$ a 1-connected closed subgroup of $G$ by \cite[Prop. 11.2.15]{HN12}. 
	Then the quotient $Q:=G/P$ is a 1-connected solvable Lie group whose Lie algebra is $\gg/\pg=\qg$. 
	Moreover, we have an isomorphism of Lie groups $Q\simeq \RR^n\rtimes_A\RR$ 
	(see also \cite[Prop. 11.1.19]{HN12}).  
	It then follows by \cite[Th. 5.5(i)]{BB21} that the Lie group $Q$ is not type~$\I$ if $S_A$ is not closed in $\RR$. 
	Consequently, if $S_A$ is not closed in $\RR$, neither the Lie group $G$ is type~$\I$, 
	as claimed above. 
	
	As a specific example of the above construction, we take $n=4$ and 
	$$A=\begin{pmatrix}
		J & 0 \\
		0 & \theta J
	\end{pmatrix}\in M_4(\RR) \text{ with }
	J:=\begin{pmatrix}
		0 & 1 \\
		-1 & 0 
	\end{pmatrix}\in M_2(\RR)
	$$
	where $\theta\in\RR$ is irrational, while $c\in\RR\setminus\{0\}$ in \eqref{exF_eq1} is arbitrary. 
	In this case $\dim G=10$. 
	Then $S_A$ is the subgroup of $\RR$ generated by $\{1,\theta\}$, hence not closed, thus $G$ is not type~$\I$.
\end{example}

\end{document}